\documentclass[smallextended]{svjour3}
\smartqed
\usepackage[OT1]{fontenc}
\usepackage[utf8]{inputenc}
\usepackage[colorlinks]{hyperref}
\usepackage{amsmath,amssymb}
\usepackage{mathrsfs}
\usepackage[english]{babel}
\usepackage{bbm}
\usepackage{color}
\usepackage{indentfirst}
\usepackage{multirow}
\usepackage{float}
\usepackage{subfig}
\usepackage{enumerate}
\usepackage{graphicx}

\begin{document}

\newcommand\E{\mathbb{E}}
\newcommand\var{\mathop{\text{Var}}}
\newcommand\cov{\mathop{\text{Cov}}}
\newcommand\tr{\mathop{\text{tr}}}
\newcommand{\Sum}[2]{\sum_{\scriptstyle #1 \atop \scriptstyle #2}}
\newcommand{\rom}[1]{\uppercase\expandafter{\romannumeral #1}}
\newcommand{\g}[2]{g_{(#1)}(#2)}
\newcommand{\f}[2]{f_{(#1)}(#2)}
\newcommand{\summ}[2]{\mathop{\sum}_{#1}^{#2}}
\newcommand\re[1]{{\color{black}#1}}
\newcommand\bl[1]{{\color{blue}#1}}

  \title{Moment approach for singular values distribution of a large auto-covariance matrix}

\titlerunning{Singular values distribution of a large auto-covariance matrix}

\author{Wang Qinwen         \and
        Yao Jianfeng
}

\institute{Wang Qinwen \at
              Department of Mathematics\\
              Zhejiang University \\
              email:wqw8813@gmail.com           
           \and
           Yao Jianfeng \at
              Department of Statistics and Actuarial Science\\
      The University of Hong Kong\\
      Pokfulam,
      Hong Kong \\
      email:jeffyao@hku.hk           
}
\date{Received: date / Accepted: date}

\maketitle

  \begin{abstract}
 Let $(\varepsilon_{t})_{t>0}$ be a sequence of independent   real random vectors of $p$-dimension   and let $X_T= \sum_{t=s+1}^{s+T}\varepsilon_t\varepsilon^T_{t-s}/T$ be the lag-$s$ ($s$ is a fixed positive integer) auto-covariance matrix of $\varepsilon_t$. Since $X_T$ is not symmetric, we consider its singular values, which are the square roots of the eigenvalues of $X_TX^T_T$. Therefore, the purpose of this paper  is to  investigate the limiting behaviors of the eigenvalues of $X_TX^T_T$ in two aspects. First, we show that the empirical spectral distribution of its eigenvalues converges  to a nonrandom limit $F$. Second, we establish the convergence of its largest eigenvalue to the right edge of $F$. Both  results are derived using moment methods.
\keywords{Auto-covariance matrix \and Singular values  \and Limiting spectral distribution \and Stieltjes transform \and Largest eigenvalue \and Moment method}
 \subclass{ 15A52 \and  60F15}

  \end{abstract}

\section{Introduction}
Let $(\varepsilon_{t})_{t>0}$ be a sequence of independent   real random vectors of $p$-dimension   and let $X_T= \sum_{t=s+1}^{s+T}\varepsilon_t\varepsilon^T_{t-s}/T$ be the lag-$s$ ($s$ is a fixed positive integer) auto-covariance matrix of $\varepsilon_t$.
The motivation of the above set up  is due to the study of dynamic factor model,  see \cite{LY012}. Set
\begin{align}
x_t=\Lambda f_t+\varepsilon_t+\mu~,
\end{align}
where $x_t$ is a  $p$-dimensional sequence observed at time $t$, $\{f_t\}$ a sequence of $m$-dimensional ``latent factor'' ($m\ll p$) uncorrelated with the error process $\{\varepsilon_t\}$ and $\mu \in \mathbb{R}^p$ is the general mean. Therefore, the lag-$s$ auto-covariance matrix of the time series $x_t$ can be considered as a finite rank (rank $m$) perturbation of  the lag-$s$ auto-covariance matrix of $\varepsilon_t$. Therefore, the first step is to study the base component, which is the  lag-$s$ auto-covariance matrix of the error term.
Besides, we are considering the random matrix framework, where the dimension $p$ and the sample size $T$ both tend to infinity with their ratio converging to a constant: $\lim p/T \rightarrow y>0$.

%
One of the main problems in random matrix theory is to investigate the convergence of the sequence of {\em empirical spectral distribution} $\{F^{A_n}\}$ for  a given sequence of symmetric or Hermitian random matrices $\{A_n\}$, where
\begin{align*}
F^{A_n}(x):=\frac 1p \sum_{j=1}^p\delta_{l_j}~,
\end{align*}
$l_j$ are the eigenvalues of $A_n$.
The limit distribution $F$, which is usually nonrandom, is called the {\em limiting spectral distribution} (LSD) of the sequence $\{A_n\}$. The study of spectral analysis of large dimensional random matrices dates back to the Wigner's famous semicircular law (\cite{wigner}) for Wigner matrix, which is further extended in various aspects: Mar\v{c}enko-Pastur (M-P) law (\cite{MP}) for large dimensional sample covariance matrix; and circular law for complex random matrix (\cite{girko}).
Another aspect is the bound on extreme eigenvalues. The literature dates back to   \cite{geman}, who proved the almost sure convergence of the largest eigenvalue of a sample covariance matrix under however some moment restrictions, which is later improved by \cite{yin}. For Wigner matrix, \cite{baiyin} found the sufficient and necessary condition for the almost sure convergence of its largest eigenvalue.  \cite{vu} presented an upper bound for the spectral norm of symmetric random matrices with independent entries and \cite{peche} derived the lower bound. \cite{Vershynin} studied the sharp upper bound of the spectral norm of products of random and deterministic matrices, which behave similarly to random matrices with independent entries, etc.

Notice that  lag-$0$ auto-covariance  matrix of $\varepsilon_t$ reduces to the standard sample covariance matrix $\frac 1T \sum_{t=1}^{T}\varepsilon_t\varepsilon^T_{t}$ and its property in large-dimension has been well developed in the literature. In contrast, very few is known for the lag-$s$ auto-covariance matrix $X_T$. Recent related work include \cite{liu}, \cite{li} and \cite{jin} for the LSD of the symmetrized auto-covariance matrix and \cite{wang} for its exact separation, which also ensures the convergence of its largest eigenvalue.

Since $X_T$ is not symmetric, its singular values are the square roots of the $p$ nonnegative eigenvalues of
\begin{align}\label{at}
A_T:=X_TX_T^T~.
\end{align}
Therefore, the main purpose of this paper is on the limiting behaviors of the eigenvalues of $A_T$.
First, we derive exact moment formula for its LSD. Next, this moment sequence is shown to satisfy the Carleman condition so that it determines uniquely the LSD. Then using power expansion, the Stieltjes transform of the LSD
\[
s(z):=\int \frac{1}{x-z}dF(x)
\]
can be found to
satisfy the following equation:
\begin{align*}
y^2z^2s^3(z)+y^2zs^2(z)-yzs^2(z)-zs(z)-1=0~.
\end{align*}
This result is  derived in \cite{li}  using a totally different  Stieltjes transform method.
Second, we show the  convergence of the largest eigenvalue of $A_T$  to the right end point  of the support of $F$ almost surely (trivially, this implies the convergence of the largest singular value of $X_T$).
Both these two results are derived using moment method. A distinctive feature here is that the matrix $A_T$ can be considered as the product of four matrices involving $\varepsilon_t$, new methodology is needed with respect to the existing literature on random matrix theory. In particular, we provide in Section \ref{de} some recursion formula related to the number of random walks, which further leads to our moment result, and the method may be of independent interest.

The rest of the paper is organized as follows. Preliminary introduction on the related graph theory is provided in Section \ref{graph}. Section \ref{lsd} derives the exact moment formula for the limiting spectral distribution of $A_T$, which further leads to the expression of its corresponding Stieltjes transform. Section \ref{largest} gives details of the convergence of the largest eigenvalue of $A_T$. In Section \ref{de}, we provide some techniques related to random walks for deriving some recursion formula, which further leads to the limiting moments in Section \ref{lsd}.




\section{Some graph theory}\label{graph}
In order to calculate the moments of the LSD, we need some information from graph theory.

For a pair of vectors of indexes $i=(i_1,\cdots,i_{2k})$ $(1\leq i_l \leq T, ~l\leq 2k)$ and $j=(j_1,\cdots,j_{2k})$ $(1\leq j_l \leq p, ~l\leq 2k)$, construct a graph $Q(i,j)$ in the following way. Draw two parallel lines, referred to as the I-line and J-line. Plot $i_1,\cdots,i_{2k}$ on the I-line and $j_1,\cdots,j_{2k}$ on the J-line, called the I-vertices and J-vertices, respectively. Draw $k$ down edges from $i_{2u-1}$ to $j_{2u-1}$, $k$ down edges from $i_{2u}+s$ to $j_{2u}$, $k$ up edges from $j_{2u-1}$ to $i_{2u}$, $k$ up edges from $j_{2u}$ to $i_{2u+1}+s$ (all these up and down edges are called vertical edges) and $k$ horizontal edges from $i_{2u}$ to $i_{2u}+s$, $k$ horizontal edges from $i_{2u-1}+s$ to $i_{2u-1}$ (with the convention that $i_{2k+1}=i_1$), where all the $u$'s are in the region: $1\leq u\leq k$. An example of a $Q$ graph with $k=3$ is shown in Figure \ref{sixq}.

\begin{figure}[h!]
\centering
\includegraphics[width=13cm]{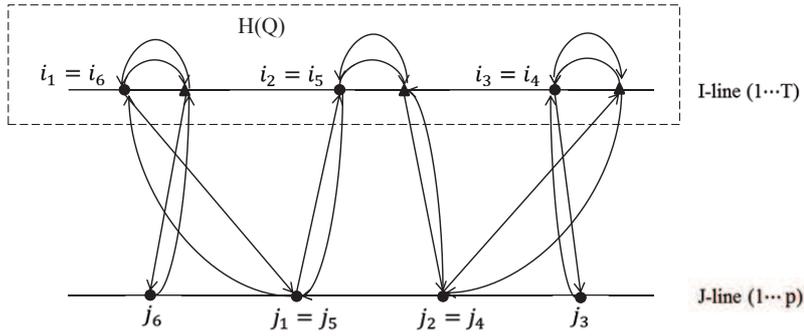}
\caption{\small An example of a $Q$ graph with $k=3$.}\label{sixq}
\end{figure}

\begin{definition}
The subgraph of all I-vertices is called the \textbf{roof} of $Q$ and is denoted by $H(Q)$ (see also Figure \ref{sixq} for illustration of $H(Q)$).
\end{definition}
\begin{definition}
The M-\textbf{minor or pillar} of $Q$ is defined as the \textbf{minor} of $Q$ by contracting all horizontal edges, which means that all horizontal edges are removed from Q and all I-vertices connected through horizontal edges are glued together. We denote the  M-\textbf{minor or pillar} of $Q$ by $M(Q)$ (see Figure \ref{sixm}).
\end{definition}

\begin{figure}[h!]
\centering
\includegraphics[width=11cm]{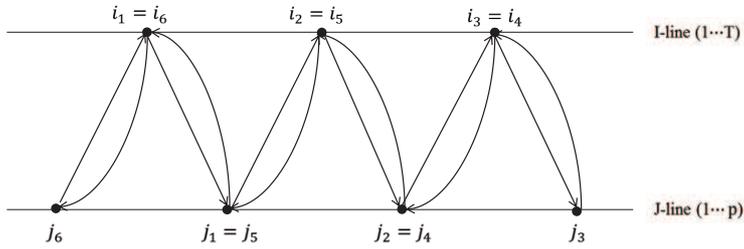}
\caption{\small The M-minor or pillar $M(Q)$ that corresponds to  the graph $Q$ in Figure \ref{sixq}.}\label{sixm}
\end{figure}

\begin{definition}
For a given $M(Q)$, glue all coincident vertical edges; namely, we regard all vertical edges with a common I-vertex and J-vertex as one edge. Then we get an undirectional connected graph. We call the resulting graph the \textbf{base} of the graph Q, and denote it by B(Q) (see Figure \ref{sixb}).
\end{definition}

\begin{figure}[h!]
\centering
\includegraphics[width=11cm]{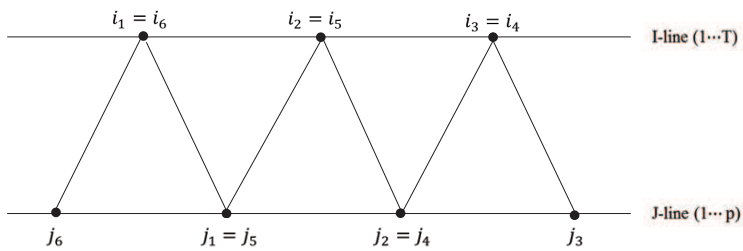}
\caption{\small The base $B(Q)$ that corresponds to  the graph $Q$ in Figure \ref{sixq}.}\label{sixb}
\end{figure}

\begin{definition}
For a vertical edge $e$ of $B(Q)$, the number of up (down) vertical edges of $Q$ coincident with $e$ is called the \textbf{up (down) multiplicity} of $e$.
\end{definition}

\begin{definition}
The \textbf{degree} of a vertex $i_l$  is the number of edges incident to this vertex.
\end{definition}

\begin{definition}\label{inn}
An up (down) \textbf {innovation} is an up (down) vertical edge which leads to a   vertex that  has not been  visited before.
\end{definition}

\begin{definition}
Two graphs are said to be \textbf{isomorphic} if one becomes the other by a suitable permutation on $(1,\cdots, T)$ and a suitable permutation on $(1,\cdots, p)$.
\end{definition}

\begin{definition}
Define a \textbf{characteristic sequence} as $(d_1~u_1~\cdots~d_{2k}~u_{2k})$, where  $\{u_1,\cdots,u_{2k}\}$ and $\{d_1,\cdots,d_{2k}\}$ are associated with the $2k$ up edges and $2k$ down edges of a pillar $M(Q)$ according to the following rule:
\begin{align*}
u_l=\left\{\begin{array}{cl}
1,& \text{the $l$-th up edge is an up innovation}~,\\
0,& \text{otherwise}~,
\end{array}\right.
\end{align*}
and
\begin{align*}
d_l=\left\{\begin{array}{cl}
0,& \text{the $l$-th down edge is an down innovation}~,\\
-1,& \text{otherwise}~.
\end{array}\right.
\end{align*}
\end{definition}
An example of the characteristic sequence associated with the pillar in Figure \ref{sixm} is given in the following Figure \ref{cha} with
\begin{align*}
&\{u_1,u_2,u_3,u_4,u_5,u_6\}=\{\text{1,~1,~0,~0,~0,~0}\}~,\\
&\{d_1,d_2,d_3,d_4,d_5,d_6\}=\{\text{0,~0,~0,~-1,~-1,~0}\}~;
\end{align*}
that is, the corresponding characteristic sequence is (0~1~0~1~0~0~-1~0~-1~0~0~0). Conversely, it can be verified that any  characteristic sequence $(d_1~u_1~\cdots~ d_{2k}~u_{2k})$ uniquely defines a pillar $M(Q)$.

\begin{figure}[h!]
\centering
\includegraphics[width=12cm]{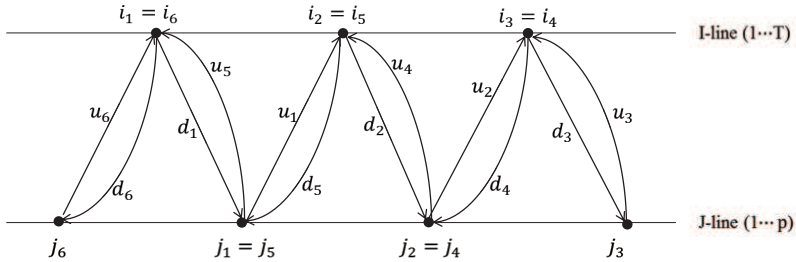}
\caption{\small Characteristic  sequence associated with the pillar $M(Q)$ in Figure \ref{sixm}.}\label{cha}
\end{figure}

\section{LSD of $A_T$ using moment method}\label{lsd}
The problem of showing  the convergence of the ESD of $A_T$ reduces to showing that the sequence of its moments $m_k(A_T):=\tr A_T^k/p$, $k\geq 1$, tends to a limit $(m_k)_k$, and this limit determines properly a probability distribution. For example, the  later property is guaranteed if the moment sequence $(m_k)_k$ satisfies the Carleman  condition:
\begin{align}\label{cal}
\sum_{k=1}^{\infty}m_{2k}^{-1/2k}=\infty~.
\end{align}

The following theorem gives the exact formula for the limiting moments $(m_k)_k$.

\begin{theorem}\label{th1}
Suppose the following conditions hold:

\noindent(a).  $(\varepsilon_{t})_t$ is a sequence of independent $p$-dimensional real valued random vectors with independent entries $\varepsilon_{it}$, $i=1,\cdots, p$ satisfying
\[
\E (\varepsilon_{ij})=0,~\E (\varepsilon_{ij}^2)=1~
\]
 and for any $\eta>0$,
\[
\frac{1}{\eta^4 pT}\sum_{i,j}\E \left(|\varepsilon_{ij}|^4I_{(|\varepsilon_{ij}|\geq \eta T^{1/4})}\right)\longrightarrow 0~\quad \text{as}~ p, T \rightarrow \infty;
\]

\noindent (b). $p$ and $T$ tend to infinity proportionally, that is,
\[
p \rightarrow \infty, ~T \rightarrow \infty,~ y_T:=p/T \rightarrow y \in (0,\infty)~.
\]

\noindent Then, with probability one, the empirical spectral distribution $F^{A_T}$ of the matrix $A_T$ in \eqref{at} tends to a limiting distribution $F$ whose $k$-th moment is given by:
\[
m_k=\sum_{i=0}^{k-1}\frac 1k \begin{pmatrix}
                               2k \\
                               i \\
                             \end{pmatrix}\begin{pmatrix}
                                            k \\
                                            i+1 \\
                                          \end{pmatrix}y^{2k-1-i}~.
\]
\end{theorem}

\begin{remark}
Using the expression of the limiting moment above, we are able to derive that the Stieltjes transform of
$F$:
\[
s(z):=\int \frac{1}{x-z}dF(x)
\]
satisfies the following equation:
\begin{align}\label{se}
y^2z^2s^3(z)+y^2zs^2(z)-yzs^2(z)-zs(z)-1=0~,
\end{align}
which coincides with an earlier result in \cite{li} found by using the Stieltjes transform method.

Indeed, by the series expansion of the function $\frac{1}{1-x}$,  the Stieltjes transform of a LSD can be expanded using its  moments:
\begin{align*}
s(z)&=\int \frac{1}{x-z}dF(x)=-\frac 1z-\sum_{i=1}^{\infty}\frac{1}{z^{i+1}}\cdot m_i\\
&=-\frac 1z-\frac 1z\cdot\sum_{i=1}^{\infty}\frac{m_i}{z^{i}}~.
\end{align*}
Let $h(z)$ be the moment generating function of $m_i$:
\[
h(z)=\sum_{i=0}^{\infty} m_iz^i~,
\]
then the part $\sum_{i=1}^{\infty}m_i/z^{i}$ equals to $h(1/z)-1$. Therefore, we have the relationship between the Stieltjes transform $s(z)$ and the moment generating function $h(z)$:
\begin{align}\label{sh}
s(z)=-\frac 1zh\left(\frac 1z\right)~.
\end{align}
In the proof of Theorem \ref{th1}, we will see that $h$ satisfies
 the equation:
\begin{align*}
xy^2h^3(x)+x(y-y^2)h^2(x)-h(x)+1=0~,
\end{align*}
which is detailed in Appendix, see \eqref{rh}.
Let $x=1/z$ in it and combine with \eqref{sh} leads to \eqref{se}.

\end{remark}

\begin{proof}(of Theorem \ref{th1})
After truncation and centralization, see Appendix A in \cite{li}, we may assume in all the following that
\begin{align*}
|\varepsilon_{ij}|\leq \delta p^{1/4}:=M, ~\E(\varepsilon_{ij})=0, ~\var(\varepsilon_{ij})=1~.
\end{align*}
With a little bit calculation, we have
\begin{align*}
m_k(A_T)&= \frac 1p\sum_{{\bf i}=1}^T\sum_{{\bf j}=1}^p\frac{1}{T^{2k}}[\varepsilon_{j_1\,i_1}\varepsilon_{j_1\,i_2}\varepsilon_{j_2\,s+i_2}\varepsilon_{j_2\,s+i_3}\varepsilon_{j_3\,i_3}\varepsilon_{j_3\,i_4}\varepsilon_{j_4\,s+i_4}\varepsilon_{j_4\,s+i_5}\nonumber\\
&\quad\quad\quad\quad\quad\quad\cdots \varepsilon_{j_{2k-1}\,i_{2k-1}}\varepsilon_{j_{2k-1}\,i_{2k}}\varepsilon_{j_{2k}\,s+i_{2k}}\varepsilon_{j_{2k}\,s+i_1}]\nonumber\\
&=\frac{1}{pT^{2k}}\sum_{\bf i,j}E_{Q({\bf i,j})}~,
\end{align*}
where the summation runs over all $Q(i,j)$-graph of length $4k$ as defined in Section \ref{graph}, the indices in ${\bf i}=(i_1, \cdots, i_{2k})$ run over $1,2,\cdots,T$ and the indices in ${\bf j}=(j_1, \cdots, j_{2k})$ run over $1,2,\cdots,p$. See the following Figure \ref{exp} for an illustration.

\begin{figure}[h!]
\centering
\includegraphics[width=13cm]{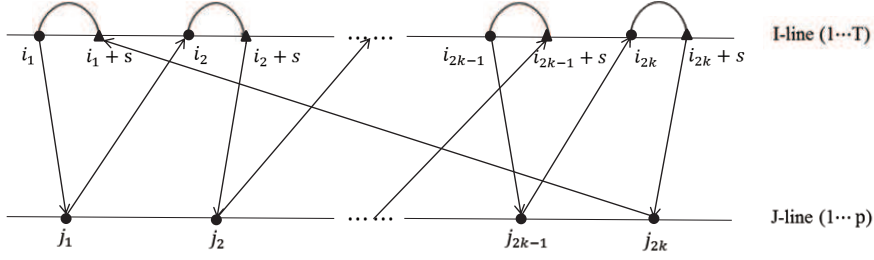}
\caption{\small The $Q(i,j)$-graph that corresponds to \eqref{etrak}.}\label{exp}
\end{figure}

Now suppose the pillar of the $Q$-graph in Figure \ref{exp} has $t$ noncoincident $I$-vertices and $s$ noncoincident $J$-vertices, which results in $s$ down innovation and $t-1$ up innovation (we make the convention that the first down edge is always a down innovation and the last up edge is not an innovation).
Then  in the corresponding characteristic sequence $(d_1~u_1~\cdots~ d_{2k}~u_{2k})$, the number of ``1'' (``1'' only appears in the even position as it corresponds to the up edge) is $t-1$, the number of ``-1'' (``-1'' only appears in the odd position) is $2k-s$ and the sequence starts and ends with  ``0''.

We classify the $Q$-graphs  in Figure \ref{exp} into three categories:

Category 1 (denoted by $Q_1$) contains all the $Q$-graphs that in its pillar $M(Q)$,  each down edge  must coincide with one and only one up edge and its base $B(Q)$ is a tree of $2k$ edges. In this category, $t+s-1=2k$ and thus $s$ is suppressed for simplicity.

Category 2 (denoted by $Q_2$) contains all the $Q$-graphs that have at least one single vertical edge.

Category 3 (denoted by $Q_3$) contains all other $Q$-graphs.

The almost sure convergence of the ESD of $A_T$ will result from the following two assertions:
\begin{align}
\E (m_k(A_T))&=\frac 1p\sum_{{\bf i}=1}^T\sum_{{\bf j}=1}^p\frac{1}{T^{2k}}\E[\varepsilon_{j_1\,i_1}\varepsilon_{j_1\,i_2}\varepsilon_{j_2\,s+i_2}\varepsilon_{j_2\,s+i_3}\varepsilon_{j_3\,i_3}\varepsilon_{j_3\,i_4}\varepsilon_{j_4\,s+i_4}\varepsilon_{j_4\,s+i_5}\nonumber\\
&\quad\quad\quad\quad\quad\quad\cdots \varepsilon_{j_{2k-1}\,i_{2k-1}}\varepsilon_{j_{2k-1}\,i_{2k}}\varepsilon_{j_{2k}\,s+i_{2k}}\varepsilon_{j_{2k}\,s+i_1}]\label{etrak}\\
&=\frac{1}{pT^{2k}}\sum_{{\bf i},{\bf j}}\E\left(E_{Q({\bf i},{\bf j})}\right)\nonumber\\
&=\sum_{i=0}^{k-1}\frac 1k \begin{pmatrix}
                               2k \\
                               i \\
                             \end{pmatrix}\begin{pmatrix}
                                            k \\
                                            i+1 \\
                                          \end{pmatrix}y_T^{2k-1-i}+o(1)~,\label{e}
\end{align}
and
\begin{align}\label{va}
&\quad~\var (m_k(A_T))\nonumber\\
&=\frac{1}{p^2T^{4k}}\sum_{{\bf i_1},{\bf j_1},{\bf i_2},{\bf j_2}}\left[\E\left(E_{Q_1({\bf i_1},{\bf j_1})}E_{Q_2({\bf i_2},{\bf j_2})}\right)-\E\left(E_{Q_1({\bf  i_1},{\bf  j_1})}\right)\E\left(E_{Q_2({\bf  i_2},{\bf  j_2})}\right)\right]\nonumber\\
&=O(p^{-2})~.
\end{align}

\noindent {\bf The proof of \eqref{e}.}
Since $\E \varepsilon_{ij}=0$, the only non-vanishing terms in \eqref{etrak} are those for which each edge in the $Q$-graph occurs at least twice. So the contribution of Category 2 is zero.

Next, we only consider those $Q$-graphs that fall in Category 1 and 3.
Denote $b_l$ the degree associated to the $I$-vertex $i_l$ $(1\leq l\leq t)$ in its corresponding $M$-pillar, then we have $b_1+\cdots+b_t=4k$, which is the total number of edges. On the other hand, since we glue the $I$-vertex $i_l$ and $i_l+s$ in the definition of $M$-pillar, each degree $b_l$ should be no less than $4$; otherwise, there will be some single vertical edges in the $Q$-graph, which results in Category 2. Therefore, we have
\begin{align*}
4k=b_1+\cdots+b_t\geq 4t~,
\end{align*}
which is $t\leq k$.

\medskip
\noindent\underline{Category 1:}
In Category 1, since in the $Q$-graph, each down edge  must coincide with one and only one up edge, the total number of non-coincident edges is $2k$.  Besides, due to the restriction that the base $B(Q)$ is a tree, we   have $s+t-1=2k$ (according to the definition of a tree that $\#\{\text{edges}\}=\#\{\text{vertices}\}-1$), and this means that the summation over the elements in the corresponding characteristic sequence is zero (it is because we have the number of ``1'' equals $t-1$ and the number of ``-1'' equals $2k-s$).

The characteristic sequence of a $M$-pillar whose corresponding $Q$-graph lies in Category 1 has the following features (see Remark \ref{re1} and \ref{re2}):
\begin{enumerate}[{(1)}]
\item The total length of the characteristic sequence is $4k$;
  \item The sequence starts with a zero and ends with a zero (the first down edge is a down innovation and the last up edge is not an innovation);
  \item The number ``1'' appears only in the even position in the sequence and the number ``-1'' only in the odd position (down edges are in the odd position while up edges are in the even);
  \item The total number of ``1'' in  the characteristic sequence is $t-1$ ($t-1$ up innovation), which equals the total number of ``-1'';
  \item The sequences are made with  the following subsequence structure:
\begin{align}\label{sub}
& 1\underbrace{00}_{\text{two}}-1\nonumber\\
&1\underbrace{000000}_{\text{six}}-1\nonumber\\
&1\underbrace{0000000000}_{\text{ten}}-1\nonumber\\
&\quad\quad\cdots\nonumber\\
&1\underbrace{0000\cdots 00000}_{4k-6}-1~.
\end{align}
\end{enumerate}

Denote $f_{t-1}(k)$ as the number of $Q$-graphs whose $M$-pillar  satisfies the above condition (1)-(5) (here, we have two index: $t-1$, which is the number of up innovations and $k$, which is a quarter of the total length of the sequence), then we have the contribution of Category 1 to \eqref{etrak}:
\begin{align}\label{m1}
&\quad~\frac{1}{pT^{2k}}\cdot\sum_{t=1}^kT(T-1)\cdots (T-t+1)p(p-1)\cdots (p-s+1)f_{t-1}(k)\nonumber\\
&=\sum_{t=1}^k y_T^{2k-t}f_{t-1}(k)+O(\frac 1p)~.
\end{align}

\noindent\underline{Category 3:} Category 3 consists two situations, see the following lemma.

\begin{lemma}[\cite{bai}]\label{lemmabai}
Denote the coincident multiplicities of the $l$-th noncoincident vertical edge by $a_l$, $l=1,2,\cdots, m$, where $m$ is the number of noncoincident vertical edges.
If $Q \in Q_3$, then (a) either there is a $a_l\geq 3$ with $t+s-1\leq m <2k$  or (b) all $a_l=2$ with $t+s-1<m=2k$.
\end{lemma}

First, we see the contribution of (a).  By the moment assumption that the moment $\E |\varepsilon_{ij}|^a$ is bounded by $M^{a-2}$ for $a\geq 2$ and $a_1+\cdots +a_m=4k$, we conclude that the expectation
\begin{align*}
\E[\varepsilon_{j_1\,i_1}\varepsilon_{j_1\,i_2}\varepsilon_{j_2\,s+i_2}\varepsilon_{j_2\,s+i_3}\varepsilon_{j_3\,i_3}\varepsilon_{j_3\,i_4}\varepsilon_{j_4\,s+i_4}\varepsilon_{j_4\,s+i_5}\cdots \varepsilon_{j_{2k}\,s+i_{2k}}\varepsilon_{j_{2k}\,s+i_1}]
\end{align*}
in \eqref{etrak} has magnitude at most $M^{4k-2m}$. Then we have \eqref{etrak}  bounded by
\begin{align}\label{b1}
&\quad \frac{1}{pT^{2k}}\cdot\sum_{t=1}^kT^tp^sM^{4k-2m}\#\{\text{isomorphism class in}~ Q_3\}\nonumber\\
&=O\left(\sum_{t=1}^kp^{t+s-k-m/2-1}\cdot \delta^{4k-2m}\right)~,
\end{align}
where the  equality is due to the fact that $p$ and $T$ are in the same order and also for fixed $k$, the part $\#\{\text{isomorphism class in}~ Q_3\}$ is of order $O(1)$. Then
\begin{align}\label{q1}
\sum_{t=1}^kp^{t+s-k-m/2-1}\leq \sum_{t=1}^kp^{m-k-m/2}\leq \sum_{t=1}^kp^{\frac{2k-1}{2}-k}=O(p^{-1/2})~,
\end{align}
which is due to the assumption that $t+s-1\leq m <2k$, so the contribution of \eqref{b1} is $o(p^{-1/2})$.

Next, we consider the contribution of (b). Since all $a_l=2$, we have the part of expectation equals 1.  Therefore, \eqref{etrak} is bounded by
\begin{align}\label{q2}
\frac{1}{pT^{2k}}\cdot\sum_{t=1}^kT^tp^s\#\{\text{isomorphism class in}~ Q_3\}=O\left(\sum_{t=1}^kp^{t+s-2k-1}\right)=O(1/p)~,
\end{align}
where the equation is due to the fact that $t+s\leq 2k$.

Therefore, combine \eqref{m1}, \eqref{q1} and \eqref{q2}, we finally have
\begin{align}\label{tr}
\E (m_k(A_T))=\sum_{t=1}^k y_T^{2k-t}f_{t-1}(k)+o(1)~.
\end{align}
To end the proof of \eqref{e}, we need to determine the value of $f_{t-1}(k)$. This involves complex combinatorics and analytic arguments and the details of the derivation is given in Section \ref{de}. Finally, using
 \eqref{fmkfor} derived in Remark \ref{fff} that
\begin{align*}
f_m(k)=\frac 1k\begin{pmatrix}
                 2k \\
                 m \\
               \end{pmatrix}\begin{pmatrix}
                              k \\
                              m+1 \\
                            \end{pmatrix}~,
\end{align*}
 we have
\[\E (m_k(A_T))=\sum_{i=0}^{k-1}\frac 1k \begin{pmatrix}
                               2k \\
                               i \\
                             \end{pmatrix}\begin{pmatrix}
                                            k \\
                                            i+1 \\
                                          \end{pmatrix}y_T^{2k-1-i}+o(1)~,
\]
which is \eqref{e}.

\bigskip
\noindent {\bf The proof of \eqref{va}.}
Recall
\begin{align*}
&\quad~\var (m_k(A_T))\nonumber\\
&=\frac{1}{p^2T^{4k}}\sum_{\bf{i_1},\bf{j_1},\bf{i_2},\bf{j_2}}\left[\E\left(E_{Q_1(\bf{i_1},\bf{j_1})}E_{Q_2(\bf{i_2},\bf{j_2})}\right)-\E\left(E_{Q_1(\bf{i_1},\bf{j_1})}\right)\E\left(E_{Q_2(\bf{i_2},\bf{j_2})}\right)\right]~.
\end{align*}
If $Q_1$ has no edges coincident with edges of $Q_2$, then
\[
\E\left(E_{Q_1(\bf{i_1},\bf{j_1})}E_{Q_2(\bf{i_2},\bf{j_2})}\right)-\E\left(E_{Q_1(\bf{i_1},\bf{j_1})}\right)\E\left(E_{Q_2(\bf{i_2},\bf{j_2})}\right)=0
\]
by independence between $E_{Q_1}$ and $E_{Q_2}$. If $Q=Q_1\bigcup Q_2$ has an overall single edge, then
\[
\E\left(E_{Q_1(\bf{i_1},\bf{j_1})}E_{Q_2(\bf{i_2},\bf{j_2})}\right)=\E\left(E_{Q_1(\bf{i_1},\bf{j_1})}\right)\E\left(E_{Q_2(\bf{i_2},\bf{j_2})}\right)=0~,
\]
so the contribution to
$\var (m_k(A_T))$ is also zero.

Now, suppose $Q=Q_1\bigcup Q_2$ contains no single edges, then there is at least one edge in $Q$ with multiplicity greater than or equal to 4, so  the number of non-coincident  I-vertices in $Q$ is at least $t_1+t_2-1$  and J-vertices is $s_1+s_2-1$. Since  $t_1+s_1-1\leq2k$ and $t_2+s_2-1\leq2k$, we have
\begin{align*}
&\quad~\var (m_k(A_T))\nonumber\\
&=\frac{1}{p^2T^{4k}}\sum_{\bf{i_1},\bf{j_1},\bf{i_2},\bf{j_2}}\left[\E\left(E_{Q_1(\bf{i_1},\bf{j_1})}E_{Q_2(\bf{i_2},\bf{j_2})}\right)-\E\left(E_{Q_1(\bf{i_1},\bf{j_1})}\right)\E\left(E_{Q_2(\bf{i_2},\bf{j_2})}\right)\right]\\
&=O\left(\frac{1}{p^2T^{4k}}T^{t_1+t_2-1}p^{s_1+s_2-1}\right)\\
&=O(p^{-2})~,
\end{align*}
which is \eqref{va}.

\bigskip
\noindent {\bf Carleman condition.}
In \cite{li},  the density function that corresponds to the Stieltjes transform $s(z)$ in \eqref{se} has been derived and it has compact support $[a,b]$, where
\begin{align}
&a=\frac 18\left(-1+20y+8y^2-(1+8y)^{3/2}\right)\cdot\mathbbm{1}_{\{y\geq 1\}}~,\nonumber\\
&b=\frac 18\left(-1+20y+8y^2+(1+8y)^{3/2}\right)~.\label{b}
\end{align}
Therefore, we have
\begin{align}\label{mkb}
m_k=\sum_{i=0}^{k-1}\frac 1k \begin{pmatrix}
                               2k \\
                               i \\
                             \end{pmatrix}\begin{pmatrix}
                                            k \\
                                            i+1 \\
                                          \end{pmatrix}y^{2k-1-i}\leq b^k~.
\end{align}
From this, it is easy to see that the Carleman condition \eqref{cal} is satisfied.

\noindent The proof of Theorem \ref{th1} is complete.
\end{proof}

\begin{remark}
For the verification of Carleman condition, it would be enough to use the Stirling's formula in $m_k$ to derive a less sharp bound $m_k\leq A^k$ for some $A\geq b$. Since the sharp bound \eqref{mkb} will also be used in Section \ref{largest}, its early introduction is thus preferred.
\end{remark}

\begin{remark}\label{re1}
The explanation of (5) in \eqref{sub} is that in the original $Q$-graph, each vertical edge  is repeated exactly twice, and then we  glue the $I$-vertex $i_l$ and $i_l+s$ in its pillar $M(Q)$, therefore, the degree of each $I$-vertex is multiple of four, which implies that the length of each subsequence in \eqref{sub} is multiple of four.
\end{remark}

\begin{remark}\label{re2}
Note that in a characteristic sequence,   subsequences in \eqref{sub} cannot intersect  each other; for example, if we arrange two subsequences 1~0~0~-1 and 1~0~0~0~0~0~0~-1 in the characteristic sequence of length $4k=16$, then the following two structures are allowed:
\begin{align*}
\begin{array}{ll}
\text{0~\re{1}~0~0~0~0~0~0~\re{-1}~\bl{1}~0~0~\bl{-1}~0~0~0} & \text{(two subsequences are parallel)}~,\\
\text{0~\re{1}~0~\bl{1}~0~0~\bl{-1}~0~\re{-1}~0~0~0~0~0~0~0} & \text{(one is completely contained in another)}~;
\end{array}
\end{align*}
while
\begin{align*}
\begin{array}{ll}
\text{0~\re{1}~0~0~0~0~0~\bl{1}~\re{-1}~0~\bl{-1}~0~0~0~0~0} & \text{(two subsequences intersect each other)}
\end{array}
\end{align*}
is not.
\end{remark}

\section{Convergence of the largest eigenvalue of $A_T$}\label{largest}
Recall that due to \eqref{mkb} in the previous section, we have the following:
\begin{align*}
\E (m_k(A_T))&=\sum_{i=0}^{k-1}\frac 1k \begin{pmatrix}
                               2k \\
                               i \\
                             \end{pmatrix}\begin{pmatrix}
                                            k \\
                                            i+1 \\
                                          \end{pmatrix}y_T^{2k-1-i}+o(1)\\
&\leq b(y_T)^k+o(1)
\end{align*}
for bounded $k$, where $b(y_T)$ is the value of $b$ in \eqref{b} while  substituting $y_T$ for $y$. The main point in this section is to improve  this estimation in order to allow a {\em growing} $k$ such that:
\begin{align}\label{ae}
\E (m_k(A_T))&=\sum_{i=0}^{k-1}\frac 1k \begin{pmatrix}
                               2k \\
                               i \\
                             \end{pmatrix}\begin{pmatrix}
                                            k \\
                                            i+1 \\
                                          \end{pmatrix}y_T^{2k-1-i}\cdot(1+o_k(1))~,
\end{align}
where this $o_k(1)$ now (depending on $k$)  tends to zero when $k\rightarrow \infty$.
For a moment, suppose we have already got \eqref{ae}, then we are able to control the largest eigenvalue of $A_T$, see the following theorem.

\begin{theorem}\label{cb}
Under the same conditions as in Theorem \ref{th1}, the largest eigenvalue of $A_T$ converges to the right endpoint $b$ defined in \eqref{b} almost surely.
\end{theorem}

\begin{proof}(of Theorem \ref{cb})
First we show that almost surely,
\begin{align}\label{ll1}
\liminf l_1 \ge b~.
\end{align}
Indeed on the set $\{\liminf l_1<b\}$, we have $\liminf l_1<b-\delta$ for some $\delta=\delta(\omega)>0$.
Let $g:\mathbb{R}\rightarrow \mathbb{R}_{+}$ be a continuous and positive function supported on $[b-\delta,b]$, with $\int g(x)dF(x)=1$, where  $F$ is the LSD of $F^{A_T}$.  Then
\begin{align}
\liminf \int g(x)dF^{A_T}(x) \le 0~.
\end{align}
For such $\omega$, $F^{A_T}$ will not converge weakly to $F$. Since this convergence occurs almost surely by Theorem \ref{th1}, the Claim \eqref{ll1} is proved.

Next, we claim that for any $\delta>0$,
\begin{align}\label{ll2}
\sum_{p=1}^\infty P(l_1>b+\delta)<\infty~.
\end{align}
Write
\[
l_1-b=(l_1-b)^{+}-(l_1-b)^{-}~,
\]
with its positive and negative parts. Claim \eqref{ll2} implies that almost surely, $(l_1-b)^{+} \to 0$. Therefore, a.s.
\[
\limsup (l_1-b) \le \limsup (l_1-b)^{+}=0~.
\]
Combine with \eqref{ll1}, we have a.s. $\lim (l_1-b)=0$.
\noindent It remains to prove Claim \eqref{ll2}.

Since we  have
\begin{align}\label{l}
&\quad ~P(l_1>b+\delta)\leq P(\tr A_T^k\geq (b+\delta)^k)\leq\frac{\E \tr A_T^k}{(b+\delta)^k}
=\frac{p\cdot\E(m_k(A_T))}{(b+\delta)^k}~,
\end{align}
and by \eqref{ae},
\begin{align*}
p\cdot\E (m_k(A_T))&=p\sum_{i=0}^{k-1}\frac 1k \begin{pmatrix}
                               2k \\
                               i \\
                             \end{pmatrix}\begin{pmatrix}
                                            k \\
                                            i+1 \\
                                          \end{pmatrix}y_T^{2k-1-i}\cdot(1+o_k(1))\\
&\leq \left(p^{1/k}b(y_T)\right)^k\cdot(1+o_k(1))\xrightarrow{p\rightarrow \infty} b^k~,
\end{align*}
where the last convergence is due  to the choice of $k$ that $k=(\log p)^{1.01}$ (see Theorem \ref{largek}), so the term  $p^{1/k}\rightarrow 1$ when $p \rightarrow \infty$.
Once we fixed this $\delta>0$ in \eqref{l}, then the right hand side  tends to $\left(\frac{b}{b+\delta}\right)^k$, which is summable.

\noindent The proof of Theorem \ref{cb} is complete.
\end{proof}

%
%

The remaining of the section is devoted to  deriving \eqref{ae} used in the proof above.

\begin{theorem}\label{largek}
Let the conditions in Theorem \ref{th1} hold, and $k$ is an integer of size $k=(\log p)^{1.01}$ (say). Then we have
\begin{align*}
\E (m_k(A_T))&=\sum_{i=0}^{k-1}\frac 1k \begin{pmatrix}
                               2k \\
                               i \\
                             \end{pmatrix}\begin{pmatrix}
                                            k \\
                                            i+1 \\
                                          \end{pmatrix}y_T^{2k-1-i}\cdot(1+o_k(1))~.
\end{align*}
\end{theorem}

\begin{proof}(of Theorem \ref{largek})
When $k \rightarrow \infty$, the term $\#\{\text{isomorphism class}\}$ in \eqref{b1} is no more a constant order. Then the main task is to show that the contribution of Category 3 to the $k$-th moment of $A_T$ when $k\rightarrow \infty$ can still be negligible compared with Category 1.

Since in Category 3, we have $t+s-1<2k$.  Using the previous notion of the characteristic sequence, we have $\#\{1\}=t-1$ and $\#\{-1\}=2k-s$, and this is equivalent to the fact that $\#\{1\}<\#\{-1\}$.

First, we consider the case that $t=1$, which is to say  $\#\{1\}=0$ and $\#\{-1\}=2k-s$. The way we choose $2k-s$ positions from the total $2k$ (the total number of length is $4k$, only the odd ones are allowed for  ``-1'') is
$$\begin{pmatrix}
     2k \\
     2k-s \\
   \end{pmatrix}.
$$
Then since we have $s$ noncoincident $J$-vertices on the $J$-line, the noncoincident vertical edges is at most $2s$ (since we have each edge repeated at least twice). Therefore, the expectation
\begin{align*}
\E[\varepsilon_{j_1\,i_1}\varepsilon_{j_1\,i_2}\varepsilon_{j_2\,s+i_2}\varepsilon_{j_2\,s+i_3}\varepsilon_{j_3\,i_3}\varepsilon_{j_3\,i_4}\varepsilon_{j_4\,s+i_4}\varepsilon_{j_4\,s+i_5}\cdots \varepsilon_{j_{2k}\,s+i_{2k}}\varepsilon_{j_{2k}\,s+i_1}]
\end{align*}
is bounded by $M^{4k-2s}$.
Therefore, we have the contribution to \eqref{etrak}:
\begin{align}\label{11}
&\quad \quad\frac{1}{pT^{2k}}\sum_{s=1}^{2k-1}Tp^sM^{4k-2s}\begin{pmatrix}
                                                 2k \\
                                                 2k-s \\
                                               \end{pmatrix}\nonumber\\
&=\left(\frac{\sqrt p}{\delta^2}\right)^{2k}p^{k-1}T^{1-2k}\delta^{4k}\sum_{s=1}^{2k-1}\left(\frac{\sqrt p}{\delta^2}\right)^{s-2k}\begin{pmatrix}
                                                                           2k \\
                                                                           s \\
                                                                         \end{pmatrix}\nonumber\\
&=\left(\frac{\sqrt p}{\delta^2}\right)^{2k}p^{k-1}T^{1-2k}\delta^{4k}\sum_{i=1}^{2k-1}\left(\frac{\delta^2}{\sqrt p}\right)^{i}\begin{pmatrix}
                                                                           2k \\
                                                                           i \\
                                                                         \end{pmatrix}\nonumber\\
&\leq\left(\frac{\sqrt p}{\delta^2}\right)^{2k}p^{k-1}T^{1-2k}\delta^{4k}\sum_{i=1}^{2k-1}\left(\frac{2k\delta^2}{\sqrt p}\right)^{i}\nonumber\\
&=O\left(\left(\frac{\sqrt p}{\delta^2}\right)^{2k}p^{k-1}T^{1-2k}\delta^{4k}\frac{2k\delta^2}{\sqrt p}\right)\nonumber\\
&=O\left(\frac{2k\delta^2}{\sqrt p}y^{2k-1}\right)~,
\end{align}
which is due to the fact that $k=(\log p)^{1.01}$.

Then consider the case that $t>1$. Since $\#\{1\}=t-1<\#\{-1\}=2k-s$, we can first construct a characteristic sequence that satisfies (1)(2)(3)(4)(5) in \eqref{sub}. Therefore, we have the degree of each $I$-vertex at least four and each edge in the $Q$-graph repeated exactly twice, which ensures that the $Q$-graph will not fall in Category 2. And the possible ways for constructing such a characteristic sequence is $f_{t-1}(k)$ by definition.
Since in the characteristic sequence, $2(t-1)$ positions have been taken to place the ``1'' and ``-1'',   there leaves $4k-2(t-1)-2/2$ (the sequence starts and ends with a zero, so we should exclude the two positions at the beginning and at the end, and  also ``-1'' appears in the odd positions, so we should divide it by two) positions to place the remaining ``-1'', whose number is $2k-t-s+1$, so the choice  is bounded by
$$
\begin{pmatrix}
  2k-t \\
  2k-t-s+1 \\
\end{pmatrix}.
$$

Let $m$ be the number of noncoincident vertical edges, which is no less than $t+s-1$, see Lemma \ref{lemmabai}, then the expectation \begin{align*}
\E[\varepsilon_{j_1\,i_1}\varepsilon_{j_1\,i_2}\varepsilon_{j_2\,s+i_2}\varepsilon_{j_2\,s+i_3}\varepsilon_{j_3\,i_3}\varepsilon_{j_3\,i_4}\varepsilon_{j_4\,s+i_4}\varepsilon_{j_4\,s+i_5}\cdots \varepsilon_{j_{2k}\,s+i_{2k}}\varepsilon_{j_{2k}\,s+i_1}]
\end{align*}
is bounded by $M^{4k-2m}\leq M^{4k-2(t+s-1)}$. Finally, the contribution to \eqref{etrak} is bounded by:
\begin{align}\label{22}
&\quad \quad\frac{1}{pT^{2k}}\sum_{s}\sum_{t}M^{4k-2(t+s-1)}f_{t-1}(k)T^tp^s\begin{pmatrix}
                                                              2k-t \\
                                                              s-1 \\
                                                            \end{pmatrix}\nonumber\\
&=\frac{1}{pT^{2k}}\sum_{t=1}^kM^{4k-2t+2}f_{t-1}(k)T^t\sum_{s=1}^{2k-t}\left(\frac{p}{M^2}\right)^s\begin{pmatrix}
                                                                                                      2k-t \\
                                                                                                      s-1 \\
                                                                                                    \end{pmatrix}\nonumber\\
&=\frac{1}{pT^{2k}}\sum_{t=1}^kM^{4k-2t+2}f_{t-1}(k)T^t\left(\frac{p}{M^2}\right)^{2k-t}\sum_{s=1}^{2k-t}\left(\frac{p}{M^2}\right)^{s-2k+t}\begin{pmatrix}
                                                                                                      2k-t \\
                                                                                                      s-1 \\
                                                                                                    \end{pmatrix}\nonumber\\
&=\frac{1}{pT^{2k}}\sum_{t=1}^kM^{4k-2t+2}f_{t-1}(k)T^t\left(\frac{p}{M^2}\right)^{2k-t}\sum_{l=0}^{2k-t-1}\left(\frac{M^2}{p}\right)^{l}\begin{pmatrix}
                                                                                                      2k-t \\
                                                                                                      l+1 \\
                                                                                                    \end{pmatrix}\nonumber\\
&\leq\frac{1}{pT^{2k}}\sum_{t=1}^kM^{4k-2t+2}f_{t-1}(k)T^t\left(\frac{p}{M^2}\right)^{2k-t}\sum_{l=0}^{2k-t-1}\left(\frac{M^2}{p}\right)^{l}(2k-t)^{l+1}\nonumber\\
&=O\left(\frac{k}{pT^{2k}}\sum_{t=1}^kM^{4k-2t+2}f_{t-1}(k)T^t\left(\frac{p}{M^2}\right)^{2k-t}\right)\nonumber\\
&=O\left(\frac{k\delta^2}{\sqrt p}\sum_{t=0}^{k-1}\frac 1k \begin{pmatrix}
                                                             2k \\
                                                             t \\
                                                           \end{pmatrix}\begin{pmatrix}
                                                                          k \\
                                                                          t+1 \\
                                                                        \end{pmatrix}y_T^{2k-t-1}~.
\right)
\end{align}
Combine \eqref{11}, \eqref{22} with the fact that $k=(\log p)^{1.01}$ leads to the fact that the contribution of $Q_3$ to \eqref{etrak} is
\begin{align*}
\sum_{t=0}^{k-1}\frac 1k \begin{pmatrix}
                                                             2k \\
                                                             t \\
                                                           \end{pmatrix}\begin{pmatrix}
                                                                          k \\
                                                                          t+1 \\
                                                                        \end{pmatrix}y_T^{2k-t-1}\cdot o_k(1)~.
\end{align*}
The proof of Theorem \ref{largek} is complete.
\end{proof}

\section{Derivation of some master recursion formula}\label{de}
In this section, we derive some recursions that further leads to the results in Theorem \ref{th1}.

First recall the definition of $f_m(k)$, which is the number of $M$-pillars whose characteristic sequence  satisfies the following conditions:
\begin{enumerate}[{(1)}]
  \item The total length of the characteristic sequence is $4k$;
  \item The sequence starts with a zero and ends with a zero;
  \item The number ``1'' appears only in the even position in the sequence and the number ``-1'' only in the odd position;
  \item The total number of ``1'' in  the characteristic sequence is $m$;
  \item The sequences are made with  the following subsequence structure:
\begin{align*}
& 1\underbrace{00}_{\text{two}}-1\nonumber\\
&1\underbrace{000000}_{\text{six}}-1\nonumber\\
&1\underbrace{0000000000}_{\text{ten}}-1\nonumber\\
&\quad\quad\cdots\nonumber\\
&1\underbrace{0000\cdots 00000}_{4k-6}-1~.
\end{align*}
\end{enumerate}

Also, we define another $M$-pillar, whose characteristic sequence also satisfies the above condition $(1)(3)(4)(5)$, but with $(2)$ replaced by the following $(2)^{*}$:

\begin{center}
$(2)^{*}$ The sequence starts with a zero and ends with three zeros.
\end{center}

\noindent We denote $g_m(k)$ as the number of such $M$-pillar satisfying $(1)(2)^{*}(3)(4)(5)$.

\subsection{Master   recursions  related to $f_m(k)$ and $g_m(k)$}

In this subsection, we derive two master   recursions  related to $f_m(k)$ and $g_m(k)$.

Once a characteristic sequence with length $2k$ is given, we denote  $S_n$ ($1\leq n\leq2k$) as the partial sums of its first $n$ elements.  We plot the points $(n,S_n)$, then connect  $(n,S_n)$ and $(n+1,S_{n+1})$ with a straight line segment.  For example, the random walk that corresponds to the characteristic sequence (0~1~0~1~0~0~-1~0~-1~0~0~0) is in the following Figure \ref{ran}.
\begin{figure}[h!]
\centering
\includegraphics[width=10cm]{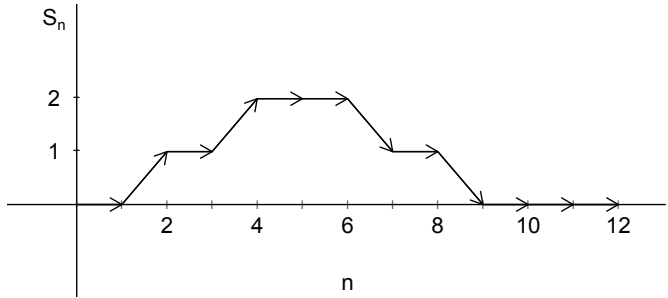}
\caption{\small The random walk that corresponds to the characteristic sequence (0~1~0~1~0~0~-1~0~-1~0~0~0).}\label{ran}
\end{figure}

\begin{definition}
We say that there is a  \textbf{return to the origin} at time $n$, if $S_n=0$.
\end{definition}
\begin{definition}
A random walk is said to have a  \textbf{first return to the origin} at time $n$, if $n>0$, and $S_{m}\neq 0$ for all $m<n$.
\end{definition}

\noindent In Figure \ref{ran}, the first return occurs at time $n=9$.
\smallskip

\subsubsection{Master recursion $1$}
First, we start with $f_m(k)$. Suppose the first return occurs at time $i$ and $\max_{0\leq n\leq i} S_n=s$, which means that in the corresponding characteristic sequence, the number of ``1'' is $s$. For the reason that the length of the subsequence structures list in (5) are all multiplicity of four and condition (2), all the  possibilities of $i$ are $i=5,7, \cdots, 4k-3, 4k-1$, we divide it into two cases:
\begin{align*}
  \begin{array}{lll}
    \text{Case ~1}: & i_1=4j-3 & 2 \leq j\leq k~,\\
    \text{Case ~2}: & i_2=4j-1 & 2 \leq j\leq k~.
  \end{array}
\end{align*}
We partition the random walk into two parts according to the first return, see Figure \ref{twoparts}.

\noindent Case 1:
First, we consider the second part, which is of length $4k-i_1=4k-4j+3$, with the number of ``1'' being $m-s$ in its corresponding characteristic sequence. But this time, the sequence may not start with a ``0'' (once it returns to the origin, it can depart immediately, and in this case, the sequence starts with a ``1''). Therefore, we add a zero in the front artificially, leading to a  total length of $4k-4j+4$, which starts and ends with a zero. So the way of constructing such a sequence is  $f_{m-s}(k-j+1)$. Then consider the first part, which has a total length of $i_1$. Suppose the first departure from the origin is at time $n-1$, where $n=2,~6,~\cdots, ~i_1-7,~i_1-3$ (also, it means that the first arrival at 1 is at the time $n$), then if we move the axes to the point $(n,1)$ (see the blue axes in Figure \ref{twoparts}, the original axes are in black) and consider the random walk above the new x-axe (see the red parts in Figure \ref{twoparts}), which has the length $i_1-n-1$. Further, this random walk starts and ends with a zero, and if we add two more zeros in its end, it will lead to a random walk with a total length of $i_1-n+1=4j-n-2$, starts with a zero and ends with three zeros, with the number of ``1'' being $s-1$. Therefore, the number of such  random walk is $g_{s-1}(\frac{4j-n-2}{4})$ according to the definition.
So we have got  the total contribution  of Case 1 is:
\begin{align}\label{c11}
\Sum{s=1, \cdots, m}{j=2, \cdots, k}\left(\sum_{n=2}^{4j-6}g_{s-1}\left(\frac{4j-n-2}{4}\right)\right)\cdot f_{m-s}(k-j+1)~.
\end{align}

\begin{figure}[h!]
\centering
\includegraphics[width=12cm]{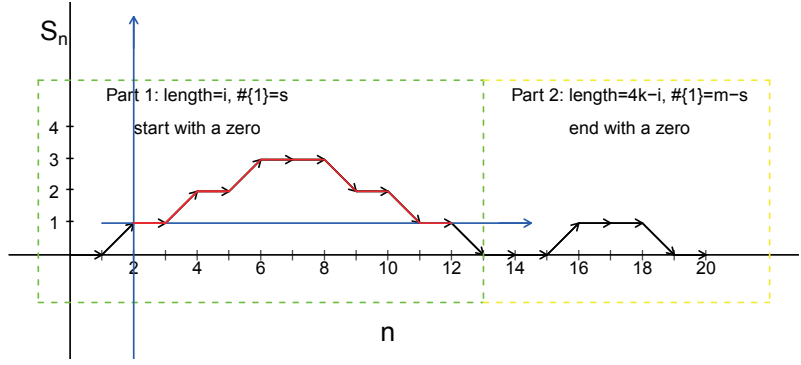}
\caption{\small Illustration of $f_m(k)$.}\label{twoparts}
\end{figure}

\noindent Case 2:
We follow the same route as in Case 1. After the first return, the remaining length is $4k-i_2=4k-4j+1$, then we add a zero in front and two zeros in the end, which leads to a random walk of total length $4k-4j+4$, starts with a zero and ends with three zeros. And the number of  ``1'' is $m-s$. By definition, the way of constructing such a random walk is $g_{m-s}(k-j+1)$. Then for the first part, also suppose the first departure from the origin is at time $n-1$ $(n=4,~8,~\cdots, ~i_2-7,~i_2-3)$ and consider the part of the random walk that is above the new x-axes (with the new origin located at $(n,1)$), which is of total length  $i_2-n-1$. Then we add two more zeros in the end, results in a random walk of total length $i_2-n+1=4j-n$, starts with a zero and ends with three zeros, and the number of ``1'' is  $s-1$. The way of constructing such a random walk is $g_{s-1}(\frac{4j-n}{4})$. And combine these two parts, the contribution of Case 2 is:
\begin{align}\label{c21}
\Sum{s=1, \cdots, m}{j=2, \cdots, k}\left(\sum_{n=4}^{4j-4}g_{s-1}\left(\frac{4j-n}{4}\right)\right)\cdot g_{m-s}(k-j+1)~.
\end{align}

\noindent Overall, combine \eqref{c11} and \eqref{c21} leads to the following recursion:
\begin{align}\label{fm}
f_{m}(k)&=\Sum{s=1, \cdots, m}{j=2, \cdots, k}\left(\sum_{n=2}^{4j-6}g_{s-1}\left(\frac{4j-n-2}{4}\right)\right)\cdot f_{m-s}(k-j+1)\nonumber\\
&\quad~+\Sum{s=1, \cdots, m}{j=2, \cdots, k}\left(\sum_{n=4}^{4j-4}g_{s-1}\left(\frac{4j-n}{4}\right)\right)\cdot g_{m-s}(k-j+1)\nonumber\\
&=\Sum{s=1, \cdots, m}{j=2, \cdots, k}\left(\sum_{n=4}^{4j-4}g_{s-1}\left(\frac{4j-n}{4}\right)\right)\cdot \left[f_{m-s}(k-j+1)+g_{m-s}(k-j+1)\right]~.
\end{align}

\subsubsection{Master recursion $2$}
Then we start with $g_m(k)$.
Also suppose the first return time is $i$, where $i=5,7,\cdots,4k-5,4k-3$ (in the definition of $g_m(k)$, the characteristic sequence ends with three zeros, so the maximum value of $i$ is $4k-3$ here). We divide  these $i$  into two parts:
\begin{align*}
  \begin{array}{lll}
    \text{Case ~1}: & i_1=4j-5 & 3 \leq j\leq k ~,\\
    \text{Case ~2}: & i_2=4j-3 & 2 \leq j\leq k~.
  \end{array}
\end{align*}
As before, we suppose the number of ``1'' is $s$ in the first part.
\smallskip

\noindent Case 1:
First for the second part, which ends with three zeros but may not start with a zero. Since the  total length is $4k-i_1=4k-4j+5$, if we add a zero in its beginning and remove the last two zeros, then it will result in a random walk whose characteristic sequence  starts and ends with a zero, whose length is $4k-4j+5+1-2=4k-4j+4$, with the number of ``1'' being $m-s$. The total number of constructing such a random walk is  $f_{m-s}(k-j+1)$. Then for the first part, suppose the first departure from the origin is at time $n-1$, where $n=4,8,\cdots, i_1-7, i_1-3$. We do the same thing as before, add the new axes whose origin is located at $(n,1)$. Then we consider the part  above this new x-axe, see also the red part in Figure \ref{gmktwoparts}, whose length is  $i_1-n-1$. We add two more zeros in the end, it actually becomes the random walk with a total length of $i_1-n+1=4j-n-4$, starts with a zero and ends with three zeros, with the number of ``1'' being $s-1$.  The way of constructing such a random walk is $g_{s-1}\left(\frac{4j-n-4}{4}\right)$. Combine all this, the contribution of Case 1 is
\begin{align}\label{c12}
\Sum{s=1, \cdots, m}{j=3, \cdots, k}\left(\sum_{n=4}^{4j-8}g_{s-1}\left(\frac{4j-n-4}{4}\right)\right)\cdot f_{m-s}(k-j+1)~.
\end{align}

\begin{figure}[h!]
\centering
\includegraphics[width=12cm]{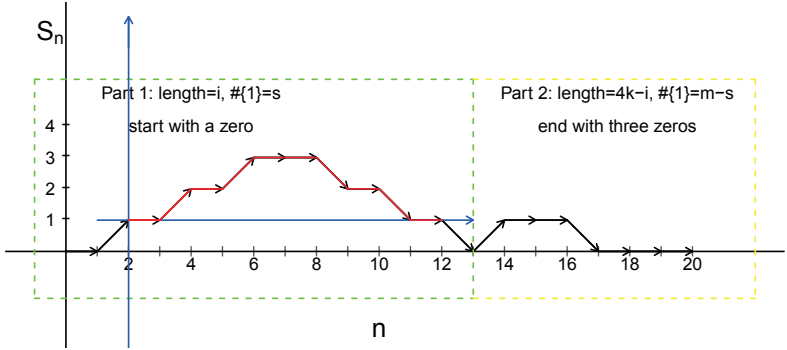}
\caption{\small Illustration of $g_m(k)$.}\label{gmktwoparts}
\end{figure}

\noindent Case 2:
For the second part, we add a zero in the front, which leads to a random walk of total length $4k-i_2+1=4k-4j+4$, starts with a zero and ends with three zeros, with the number of ``1'' being $m-s$, so the way of constructing such a random walk is $g_{m-s}(k-j+1)$.
Then for the first part, we also consider the part that above the new x-axe. Since  $n=2,6,\cdots, i_2-7,i_2-3$ this time,  we add two more zeros in the end and this leads to a random walk with total length of $i_2-n+1=4j-3-n+1$, starts with a zero and ends with three zeros, with the number of ``1'' being $s-1$, and the total way of constructing such a random walk is $g_{s-1}\left(\frac{4j-2-n}{4}\right)$.
Combine these two parts, the contribution of Case 2 is
\begin{align}\label{c22}
\Sum{s=1, \cdots, m}{j=2, \cdots, k}\left(\sum_{n=2}^{4j-6}g_{s-1}\left(\frac{4j-n-2}{4}\right)\right)\cdot g_{m-s}(k-j+1)~.
\end{align}

Combine \eqref{c12} and \eqref{c22} leads to the recursion that
\begingroup
\allowdisplaybreaks
\begin{eqnarray}\label{gm}
g_m(k)&=&\Sum{s=1, \cdots, m}{j=3, \cdots, k}\left(\sum_{n=4}^{4j-8}g_{s-1}\left(\frac{4j-n-4}{4}\right)\right)\cdot f_{m-s}(k-j+1)\nonumber\\
&&\quad+\Sum{s=1, \cdots, m}{j=2, \cdots, k}\left(\sum_{n=2}^{4j-6}g_{s-1}\left(\frac{4j-n-2}{4}\right)\right)\cdot g_{m-s}(k-j+1)\nonumber\\
&=&\Sum{s=1, \cdots, m}{j=3, \cdots, k}\left(\sum_{n=6}^{4j-6}g_{s-1}\left(\frac{4j-n-2}{4}\right)\right)\cdot (f_{m-s}(k-j+1)+g_{m-s}(k-j+1))\nonumber\\
&&\quad+\Sum{s=1,\cdots,m}{j=2,\cdots,k}g_{s-1}(j-1)g_{m-s}(k-j+1)\nonumber\\
&=&\Sum{s=1, \cdots, m}{j=3, \cdots, k}\left(\sum_{n=8}^{4j-4}g_{s-1}\left(\frac{4j-n}{4}\right)\right)\cdot (f_{m-s}(k-j+1)+g_{m-s}(k-j+1))\nonumber\\
&&\quad+\Sum{s=1,\cdots,m}{j=2,\cdots,k}g_{s-1}(j-1)g_{m-s}(k-j+1)~.
\end{eqnarray}
\endgroup

As a result, \eqref{fm} and \eqref{gm} lead to the two recursions  that related to $f_m(k)$ and $g_m(k)$:
\begin{align}\label{fg1}
&\left\{
\begin{array}{l}
 f_{m}(k)-g_{m}(k)=\displaystyle\Sum{s=1,\cdots, m}{j=2, \cdots, k} g_{s-1}(j-1)f_{m-s}(k-j+1)\\[10mm]
 g_{m}(k)=\displaystyle\Sum{s=1, \cdots, m}{j=3, \cdots, k}\left(\sum_{l=2}^{j-1}g_{s-1}(j-l)\right)\cdot \left(f_{m-s}(k-j+1)+g_{m-s}(k-j+1)\right)\\[3mm]
\end{array}
\right.\nonumber\\
 &\quad \quad\quad\quad\quad+\displaystyle\Sum{s=1,\cdots,m}{j=2,\cdots,k}g_{s-1}(j-1)g_{m-s}(k-j+1)
\end{align}

\subsection{Recursions related to  $F_k(z)$ and $G_k(z)$}
Define $F_k(z)=\sum_{m=0}^kf_{m}(k)z^m$ and $G_k(z)=\sum_{m=0}^kg_{m}(k)z^m$.
Recall the definition of $f_m(k)$ and  $g_{m}(k)$, where  $m\leq k-1$, so we make the convention that $f_k(k)=g_k(k)=0$ and $f_0(k)=g_0(k)=1$. In this subsection, the main purpose is to derive the following two recursions that related to  $F_k(z)$ and $G_k(z)$:
\begin{align}\label{fkgk}
&\left\{
\begin{array}{l}
F_k(z)-G_k(z)=z\displaystyle\sum_{j=2}^kG_{j-1}(z)\cdot F_{k-j+1}(z)\\[5mm]
G_k(z)=1+z\displaystyle\sum_{j=3}^k\sum_{l=2}^{j-1}G_{j-l}(z)\left(F_{k-j+1}(z)+G_{k-j+1}(z)\right)
\end{array}
\right.\nonumber\\
&\quad \quad \quad \quad \quad \quad +z\displaystyle\sum_{j=2}^kG_{j-1}(z)G_{k-j+1}(z)~.
\end{align}

By multiplying $z^m$ on both sides in  the first recursion in \eqref{fg1},
and then summing from $m=1$ to $k$,
the left side equals to:
\begin{align*}
\sum_{m=1}^kf_{m}(k)z^m-\sum_{m=1}^kg_{m}(k)z^m&=\sum_{m=0}^kf_{m}(k)z^m-\sum_{m=0}^kg_{m}(k)z^m\\
&=F_k(z)-G_k(z)~,
\end{align*}
and the right side equals to
\begin{align}\label{1}
z\sum_{m=1}^k\Sum{s=1,\cdots, m}{j=2, \cdots, k} g_{s-1}(j-1)z^{s-1}\cdot f_{m-s}(k-j+1)z^{m-s}~.
\end{align}
Now consider
\begin{align*}
G_{j-1}(z)\cdot F_{k-j+1}(z)&=\sum_{m=0}^{j-1}g_{m}(j-1)z^m\cdot \sum_{n=0}^{k-j+1}f_{n}(k-j+1)z^n\\
&:=\sum_{s=0}^ka_s\cdot z^s~,
\end{align*}
where
\begin{align}\label{as}
&s=m+n~,\nonumber\\
&a_s=\sum_{m+n=s}g_{m}(j-1)f_{n}(k-j+1)~.
\end{align}
Then \eqref{1} equals to
\begin{align*}
z\sum_{j=2}^k\sum_{m=1}^ka_{m-1}\cdot z^{m-1}&=z\sum_{j=2}^k\sum_{m=0}^{k-1}a_{m}\cdot z^{m}=z\sum_{j=2}^k\sum_{m=0}^{k}a_{m}\cdot z^{m}\\
&=z\sum_{j=2}^kG_{j-1}(z)\cdot F_{k-j+1}(z)~,
\end{align*}
where the second equality is due to the fact that $a_k=0$ (since in \eqref{as}, we have $m\leq j-2$ and $n \leq k-j$ by definition, then $s=m+n \leq k-2$, so the term $a_{k-1}=a_k=0$).
Therefore, we have got
\begin{align*}
F_k(z)-G_k(z)=z\sum_{j=2}^kG_{j-1}(z)\cdot F_{k-j+1}(z)~,
\end{align*}
which is the first recursion in \eqref{fkgk}.

Next, from the second recursion in \eqref{fg1}, we have:
\begin{align}\label{gmk}
g_{m}(k)&=\underbrace{\Sum{s=1, \cdots, m}{j=3, \cdots, k}\left(\sum_{l=2}^{j-1}g_{s-1}(j-l)\right)\cdot \left(f_{m-s}(k-j+1)+g_{m-s}(k-j+1)\right)}_{(\rom{1})}\nonumber\\
&+\underbrace{\Sum{s=1,\cdots,m}{j=2,\cdots,k}g_{s-1}(j-1)g_{m-s}(k-j+1)}_{(\rom{2})}~.
\end{align}

\noindent Consider
\begin{align}
G_{j-l}(z)F_{k-j+1}(z)&=\sum_{u=0}^{j-l}g_{u}(j-l)z^u\cdot \sum_{v=0}^{k-j+1}f_{v}(k-j+1)z^v\nonumber\\
&:=\sum_{n=0}^{k+1-l}b_nz^n~,
\end{align}
where $n=u+v$, $b_n=\sum_{u+v=n}g_{u}(j-l)f_{v}(k-j+1)$, also, we have $u\leq j-l-1$ and $v\leq k-j$, which leads to $n=u+v\leq k-l-1$, so the terms that correspond to $n=k-l+1,~k-l$ equal zero.
For the same reason,
\begin{align}
G_{j-l}(z)G_{k-j+1}(z)&=\sum_{u=0}^{j-l}g_{u}(j-l)z^u\cdot \sum_{v=0}^{k-j+1}g_{v}(k-j+1)z^v\nonumber\\
&:=\sum_{n=0}^{k+1-l}c_nz^n~,
\end{align}
where $n=u+v$, $c_n=\sum_{u+v=n}g_{u}(j-l)g_{v}(k-j+1)$ and  the terms that correspond to $n=k-l+1,~k-l$ equal zero.
And
\begin{align}
G_{j-1}(z)G_{k-j+1}(z)&=\sum_{u=0}^{j-l}g_{u}(j-l)z^u\cdot \sum_{v=0}^{k-j+1}g_{v}(k-j+1)z^v\nonumber\\
&:=\sum_{n=0}^{k}d_nz^n~,
\end{align}
where $n=u+v$, $d_n=\sum_{u+v=n}g_{u}(j-1)g_{v}(k-j+1)$ and  the terms that correspond to $n=k,~k-1$ equal zero since $u\leq j-2$ and $v\leq k-j$ lead to $n=u+v \leq k-2$.

We multiply $z^m$ on both sides of \eqref{gmk} and sum from $m=1$ to $k$, then the left side equals to
\begin{align}\label{left}
G_k(z)-1~.
\end{align}

Now consider the part $(\rom1)$:
\begin{align}\label{r1}
\sum_{m=1}^k(\rom1)\cdot z^m&=\sum_{m=1}^k\sum_{j=3}^k\sum_{l=2}^{j-1}(b_{m-1}+c_{m-1})\cdot z^m\nonumber\\
&=z\sum_{j=3}^k\sum_{l=2}^{j-1}\left[\sum_{n=0}^{k-1}b_nz^n+\sum_{n=0}^{k-1}c_nz^n\right]\nonumber\\
&=z\sum_{j=3}^k\sum_{l=2}^{j-1}\left[\sum_{n=0}^{k-l-1}b_nz^n+\sum_{n=0}^{k-l-1}c_nz^n\right]\nonumber\\
&=z\sum_{j=3}^k\sum_{l=2}^{j-1}\left[\sum_{n=0}^{k-l+1}b_nz^n+\sum_{n=0}^{k-l+1}c_nz^n\right]\nonumber\\
&=z\sum_{j=3}^k\sum_{l=2}^{j-1}\left[G_{j-l}(z)F_{k-j+1}(z)+G_{j-l}(z)G_{k-j+1}(z)\right]~.
\end{align}
And for the part $(\rom2)$,
\begin{align}\label{r2}
\sum_{m=1}^k(\rom2)\cdot z^m&=\sum_{m=1}^k\sum_{j=2}^kd_{m-1}\cdot z^m=z\sum_{j=2}^k\sum_{m=0}^{k-1}d_{m}\cdot z^{m}\nonumber\\
&=z\sum_{j=2}^k\sum_{m=0}^{k}d_{m}\cdot z^{m}=z\sum_{j=2}^kG_{j-1}(z)G_{k-j+1}(z).
\end{align}
Combining  \eqref{gmk}, \eqref{left}, \eqref{r1} and \eqref{r2}, we have got
\begin{align*}
G_k(z)&=1+z\sum_{j=3}^k\sum_{l=2}^{j-1}\left[G_{j-l}(z)F_{k-j+1}(z)+G_{j-l}(z)G_{k-j+1}(z)\right]\nonumber\\
&\quad  +z\sum_{j=2}^kG_{j-1}(z)G_{k-j+1}(z)~,
\end{align*}
which is the second recursion in \eqref{fkgk}.

\subsection{Equations  related to $F(z,x)$ and $G(z,x)$}
Define $F(z,x)=\sum_{k=0}^{\infty}F_k(z)x^k$, $G(z,x)=\sum_{k=0}^{\infty}G_k(z)x^k$ and the term that corresponding to $k=0$ equals $0$. In this section, we derive  the following equations:
\begin{align}\label{FG}
&\left\{
\begin{array}{l}
F(z,x)-G(z,x)=zF(z,x)G(z,x)~,\\[2mm]
G(z,x)=\displaystyle\sum_{k=1}^{\infty}x^k+z\displaystyle\sum_{k=2}^{\infty}F(z,x)G(z,x)x^{k-1}+z\displaystyle\sum_{k=2}^{\infty}G(z,x)G(z,x)x^{k-1}
\end{array}
\right.\nonumber\\
&\quad \quad \quad \quad \quad +zG(z,x)G(z,x)~,
\end{align}
which will lead to the solutions of $F(z,x)$ and $G(z,x)$ as functions of $z$ and $x$.

Since the first recursion in \eqref{fkgk}:
\begin{align*}
F_k(z)-G_k(z)=z\sum_{j=2}^kG_{j-1}(z)\cdot F_{k-j+1}(z)~,
\end{align*}
we multiply $x^k$ on both sides and do summation from $k=1$ to $\infty$, then the left side is exactly $F(z,x)-G(z,x)$.
The right side is
\begin{align*}
z\sum_{k=2}^{\infty}\sum_{j=2}^kG_{j-1}(z)F_{k-j+1}(z)x^k=zF(z,x)G(z,x)~,
\end{align*}
which leads to the first equation in \eqref{FG}:
\begin{align*}
F(z,x)-G(z,x)=zF(z,x)G(z,x)~.
\end{align*}
Then we denote the second  recursion in \eqref{fkgk} as follows:
\begin{align}
G_k(z)&=1+\underbrace{z\sum_{j=3}^k\sum_{l=2}^{j-1}G_{j-l}(z)F_{k-j+1}(z)}_{(\rom1)}+\underbrace{z\sum_{j=3}^k\sum_{l=2}^{j-1}G_{j-l}(z)G_{k-j+1}(z)}_{(\rom2)}\nonumber\\
&\quad+\underbrace{z\sum_{j=2}^kG_{j-1}(z)G_{k-j+1}(z)}_{(\rom3)}~.
\end{align}
We multiply $x^k$ on both sides and then do summation from $k=1$ to $\infty$, the left side equals  $G(z,x)$.
The part $(\rom1)$:
\begin{align}\label{p1}
\sum_{k=1}^{\infty}(\rom1)\cdot x^k&=z\sum_{k=1}^{\infty}\sum_{j=3}^k\sum_{l=2}^{j-1}G_{j-l}(z)F_{k-j+1}(z)x^k\nonumber\\
&=z\sum_{k=3}^{\infty}\sum_{l=2}^{k-1}\left(\sum_{j=l+1}^{k}G_{j-l}(z)x^{j-l}\cdot F_{k-j+1}(z)x^{k-j+1}\right)\cdot x^{l-1}\nonumber\\
&=z\sum_{l=2}^{\infty}\left(\sum_{k=l+1}^{\infty}\sum_{j=l+1}^{k}G_{j-l}(z)x^{j-l}\cdot F_{k-j+1}(z)x^{k-j+1}\right)\cdot x^{l-1}\nonumber\\
&=z\sum_{l=2}^{\infty}F(z,x)G(z,x) x^{l-1}~,
\end{align}
and for the same reason, the contribution of part $(\rom2)$ equals

\begin{align}\label{p2}
z\sum_{l=2}^{\infty}G(z,x)G(z,x) x^{l-1}~.
\end{align}
Lastly for the part $(\rom3)$:
\begin{align}\label{p3}
z\sum_{k=1}^{\infty}\sum_{j=2}^kG_{j-1}(z)G_{k-j+1}(z)x^k&=z\sum_{k=2}^{\infty}\sum_{j=2}^kG_{j-1}(z)x^{j-1}\cdot G_{k-j+1}(z)x^{k-j+1}\nonumber\\
&=zG(z,x)G(z,x).
\end{align}
Finally, combining \eqref{p1}, \eqref{p2} and \eqref{p3} leads to:
\begin{align*}
G(z,x)&=\sum_{k=1}^{\infty}x^k+z\sum_{k=2}^{\infty}F(z,x)G(z,x)x^{k-1}+z\sum_{k=2}^{\infty}G(z,x)G(z,x)x^{k-1}\nonumber\\
&\quad +zG(z,x)G(z,x)~,
\end{align*}
which is the second equation in \eqref{FG}.

\subsection{Exact formula for $m_k$}
First, since $F(z,x)=\sum_{k=0}^{\infty}F_k(z)x^k$, we have
\begin{align}\label{FK}
F_k(z)=\frac{1}{k!}\cdot \left.\frac{\partial F(z,x)}{\partial x^k}\right|_{x=0}~.
\end{align}
In the following, we will use the shorthands $F=F(z,x)$, $G=G(z,x)$ and $F_k=F_k(z)$.  In \eqref{FG}, we can first express $G$ in function of $F$ using the first equation and then derive from the second equation that
\begin{align*}
F=\sum_{l=1}^{\infty}x^l\cdot (1+z^2F^2+2zF+zF^2+z^2F^3+zF^2)~.
\end{align*}
Taking $k$-th derivative  on both sides with respect to $x$ and combining with \eqref{FK}, we have:
\begin{align}\label{fffk}
F_k&=1+2z\sum_{j=1}^{k-1}F_{j}+(z^2+2z)\sum_{l=2}^{k-1}\Sum{a+b=l}{a,b\geq 1}F_aF_b+z^2\sum_{l=3}^{k-1}\Sum{a+b+c=l}{a,b,c\geq 1}F_aF_bF_c~.
\end{align}

Due to the definition of $m_k$ that $m_k
=y^{2k-1}F_{k}\left(\frac {1}{y}\right)$, substituting $1/y$ for $z$ in \eqref{fffk} and  multiplying both sides by $y^{2k-1}$, we get the recursion for $m_k$:
\begin{align}\label{mk1}
m_k&=y^{2k-1}+\frac 2y \sum_{j=1}^{k-1}m_j\cdot y^{2k-2j}+(\frac{1}{y^2}+\frac 2y)\sum_{l=2}^{k-1}\Sum{a+b=l}{a,b\geq 1}m_am_b\cdot y^{2k-2l+1}\nonumber\\
&+\frac{1}{y^2}\sum_{l=3}^{k-1}\Sum{a+b+c=l}{a,b,c\geq 1}m_am_bm_c\cdot y^{2k-2l+2}~.
\end{align}
Then we substitute $k-1$ for $k$ in \eqref{mk1} and multiply both sides by $y^2$, which is the following equation:
\begin{align}\label{mk2}
y^2\cdot m_{k-1}&=y^{2k-1}+\frac 2y \sum_{j=1}^{k-2}m_j\cdot y^{2k-2j}+(\frac{1}{y^2}+\frac 2y)\sum_{l=2}^{k-2}\Sum{a+b=l}{a,b\geq 1}m_am_b\cdot y^{2k-2l+1}\nonumber\\
&+\frac{1}{y^2}\sum_{l=3}^{k-2}\Sum{a+b+c=l}{a,b,c\geq 1}m_am_bm_c\cdot y^{2k-2l+2}~.
\end{align}

\noindent By combining \eqref{mk1} and \eqref{mk2}, we have:
\begin{align}\label{mk3}
m_k&=(2y+y^2)m_{k-1}+(y+2y^2)\cdot\Sum{a+b=k-1}{a,b\geq 1}m_am_b+y^2\cdot\Sum{a+b+c=k-1}{a,b,c\geq 1}m_am_bm_c~.
\end{align}

By the definition of $m_k$ that $m_0=1$, we have
\begin{align}\label{mk4}
&\Sum{a+b+c=k-1}{a,b,c\geq 1}m_am_bm_c=\Sum{a+b+c=k-1}{a,b,c\geq 0}m_am_bm_c-3\Sum{a+b=k-1}{a,b\geq 1}m_am_b-3m_{k-1}~,\nonumber\\
&\Sum{a+b=k-1}{a,b\geq 1}m_am_b=\Sum{a+b=k-1}{a,b\geq 0}m_am_b-2m_{k-1}~.
\end{align}
Bringing these two equations in \eqref{mk4} into \eqref{mk3}, we get:
\begin{align}\label{rm}
y^2 \Sum{a+b+c=k-1}{a,b,c\geq 0}m_am_bm_c+(y-y^2)\Sum{a+b=k-1}{a,b\geq 0}m_am_b=m_k~.
\end{align}
Let $h(x)$ be the moment generating function: $h(x)=\sum_{k=0}^{\infty}m_k x^k$,
the we multiply $x^k$ on both sides of \eqref{rm} and do summation from $k=1$ to $\infty$ and combine with the fact that
\begin{align}\label{hx}
h(x)=1+\sum_{k=1}^{\infty}m_k x^k~,
\end{align} we have the following equality:
\begin{align}\label{rh}
xy^2h^3(x)+x(y-y^2)h^2(x)-h(x)+1=0~.
\end{align}
Based on the theory of B\"{u}rmann-Lagrange series, see Page 145 of \cite{polya}, and let $z=h(x)-1$ and $\varphi=y^2(z+1)^3+(y-y^2)(z+1)^2$, we may invert \eqref{rh} to obtain that
\begin{align*}
z=\sum_{n=1}^{\infty} \left.\frac{w^n}{n!}\left[\frac{d^{n-1}\left[y^n(z+1)^{2n}(yz+1)^n\right]}{dz^{n-1}}\right]\right|_{z=0}~,
\end{align*}
where $w=z/\varphi=x$.
Then based on the Leibniz's rule in differential calculus,
we have
\begin{align*}
\frac{d^{n-1}\left[(z+1)^{2n}(yz+1)^n\right]}{dz^{n-1}}= \sum_{i=0}^{n-1}\begin{pmatrix}
                                                                                          n-1 \\
                                                                                          i \\
                                                                                        \end{pmatrix}
\left[\frac{d^i\left[(z+1)^{2n}\right]}{dz^i}\cdot \frac{d^{n-1-i}\left[(yz+1)^n\right]}{dz^{n-1-i}}\right]~,
\end{align*}
which leads to the fact that
\begin{align*}
h(x)=1+z=1+\sum_{n=1}^{\infty}\left[\sum_{i=0}^{n-1}\frac 1n \begin{pmatrix}
                                                               2n \\
                                                               i \\
                                                             \end{pmatrix}\begin{pmatrix}
                                                                            n \\
                                                                            i+1 \\
                                                                          \end{pmatrix}y^{2n-1-i}
\right]\cdot x^n~,
\end{align*}
and this is equivalent to
\begin{align}\label{mk}
m_k=\sum_{i=0}^{k-1}\frac 1k \begin{pmatrix}
                               2k \\
                               i \\
                             \end{pmatrix}\begin{pmatrix}
                                            k \\
                                            i+1 \\
                                          \end{pmatrix}y^{2k-1-i}~.
\end{align}

\begin{remark}\label{fff}
Since
\begin{align*}
m_k=y^{2k-1}F_k\left(\frac 1y\right)=y^{2k-1}\sum_{m=0}^kf_m(k)\frac{1}{y^m}=\sum_{m=0}^{k-1}f_m(k)y^{2k-1-m}~,
\end{align*}
\eqref{mk} reduces to the fact that
\begin{align}\label{fmkfor}
f_m(k)=\frac 1k\begin{pmatrix}
                 2k \\
                 m \\
               \end{pmatrix}\begin{pmatrix}
                              k \\
                              m+1 \\
                            \end{pmatrix}~.
\end{align}
\end{remark}

\begin{remark}
The recursion \eqref{rm} has a remarkable nature. Notice that the recursion
\[
c_k=\Sum{a+b=k-1}{a,b\geq 0}c_ac_b
\]
and
\[
d_k=\Sum{a+b+c=k-1}{a,b,c\geq 0} d_ad_bd_c
\]
define the  (standard) Catalan numbers and the generalized Catalan numbers of order three, respectively (see \cite{hp}). The moment sequence $(m_k)$ of the LSD of this paper can be thought as a complex combination of these two families of Catalan numbers.
\end{remark}


\end{document}